\documentclass[12pt]{article}
\usepackage{amsmath, amssymb, amsthm}
\newcommand{\Neg}{\operatorname{Neg}}
\newcommand{\ac}[1]{{#1}^\sim}
\newcommand{\res}[1]{\overline{#1}}
\newcommand{\I}[1]{\mathcal{I}(\left.#1\right.)}
\newcommand{\U}[1]{\mathcal{U}(\left.#1\right.)}
\newcommand{\vr}[1]{\mathcal{O}_{#1}}
\newcommand{\vmin}{v_\mathrm{min}}
\newcommand{\psf}[1]{\operatorname{PSF}(#1)}
\newcommand{\Gal}[1]{\operatorname{Gal}(#1)}
\newcommand{\spn}[1]{\left< #1 \right>}
\newcommand{\gspan}[3]{\left< #3 : #1 \subset #2 \right>}
\newcommand{\support}{\operatorname{support}}
\newcommand{\coeffs}{\operatorname{coeffs}}
\newcommand{\trdeg}{\operatorname{tr deg}}
\newcommand{\id}{\operatorname{id}}
\newcommand{\ch}{\operatorname{char}}
\newcommand{\comment}[1]{}
\newcommand{\n}{\par\noindent}

\newcommand{\sn}{\par\smallskip\noindent}
\newcommand{\mn}{\par\medskip\noindent}
\newcommand{\bn}{\par\bigskip\noindent}
\newcommand{\pars}{\par\smallskip}
\newtheorem{theorem}{Theorem}[section]
\newtheorem{lemma}[theorem]{Lemma}
\newtheorem{corollary}[theorem]{Corollary}
\newtheorem{proposition}[theorem]{Proposition}

\theoremstyle{definition}
\newtheorem{definition}{Definition}

\theoremstyle{remark}
\newtheorem{remark}[definition]{Remark}
\newtheorem{example}[definition]{Example}
\newcommand{\adresse}{\par\bigskip\small \rm
F.-V. Kuhlmann:\par
Research Center for Algebra, Logic and Computation,\par
University of Saskatchewan, S7N 5E6, Canada.\par
Email: fvk@math.usask.ca
\pars
S. Kuhlmann:\par
Research Center for Algebra, Logic and Computation,\par
University of Saskatchewan, S7N 5E6, Canada.\par
Email: skuhlman@math.usask.ca
\pars
J. W. Lee:\par Department of Mathematics, Stanford University,
California, 94305-2125, USA.\par
Email: jlee@math.stanford.edu}
\title{Valuation bases for generalized algebraic
series fields%
}

\author{Franz-Viktor Kuhlmann, Salma Kuhlmann and Jonathan W. Lee%
\footnote{Research of the first two authors was supported
by NSERC Discovery Grants; of the third by an NSERC
Undergraduate Student Research Award.}}
\date{December 25, 2007}
\begin{document}
\maketitle
\begin{abstract}
We investigate valued fields which admit a valuation basis.
Given a countable ordered abelian group $G$ and
a real closed, or algebraically closed field $F$,
we give a sufficient condition for a valued
subfield of the field of generalized power
series $F((G))$ to admit a $K$-valuation basis.
We show that the field of rational functions $F(G)$ and the field
$\ac{F(G)}$ of power series in $F((G))$ algebraic over $F(G)$
satisfy this condition. It follows that for archimedean $F$
and divisible $G$ the real closed field $\ac{F(G)}$
admits a restricted exponential function.
\end{abstract}
\section{Introduction}
Before describing the motivation for this research, and stating the main
results obtained, we need to briefly remind the reader of some
terminology and background on valued and ordered fields
(see \cite{KS1} for more details).
\begin{definition}
 Let $K$ be a field and $V$ be a $K$-vector space. Let $\Gamma$ be
a totally ordered set, and $\infty$ be an element larger than any
element of $\Gamma$. A surjective map
$v: V \to \Gamma \cup \{\infty\}$
is a \emph{valuation} on $V$ if for all
$x,y \in V$ and
$r \in K$, the following holds: (i) $v(x)=\infty$ if and only if $x=0$,
(ii) $v(rx)=v(x)$ if $r \neq 0$,
(iii) $v(x-y)\geq \min\{v(x),v(y)\}$.
\end{definition}
An important example is when $G$ is an ordered abelian group.
Set $|g|:=\max\{g,-g\}$ for $g\in G$. For non-zero $g_1,g_2\in G$
say $g_1$ is archimedean equivalent to $g_2$
if there exists an integer $r$ such that
$r\, |g_1| \geq |g_2|$ and $r\,|g_2| \geq |g_1|.$
Denote by $[g]$ the equivalence class of $g \neq 0$, and by $v$
the \emph{natural valuation} on $G$, that is, $v(g):=[g]$ for $g \neq
0$, and $v(0):=\infty$.
If $G$ is divisible, then $G$ is a valued vector space over the
rationals.
\begin{definition}
We say that $\{b_i : i\in I\} \subset V$ is \emph{$K$-valuation independent}
if for all $r_i\in K$ such
that $r_i=0$ for all but finitely many $i\in I$,
\[	v\left(\sum_{i\in I} r_i b_i\right)
	= \min_{\{ i\in I \,:\, r_i \neq 0 \}} \{ v(b_i) \} \;.\]
A $K$-\emph{valuation basis} is a $K$-basis which is $K$-valuation
independent.
\end{definition}

\pars
We now recall some facts about valued fields
(see \cite{Ri} for more details).
\begin{definition}
Let $K$ be a field, $G$ an ordered abelian group and $\infty$ an
element greater than every element of $G$. \sn A surjective map
$w: K \to G \cup \{\infty\}$ is a \emph{valuation} on $K$
if for all $a,b\in K$\n (i) \ $w(a)=\infty$ if and only if $a=0$,\ \ (ii) \
$w(ab)= w(a)+w(b)$,\ \ (iii) \ $w(a-b)\geq \min\{w(a),w(b)\}$. \sn
We say that $(K,w)$ is a \emph{valued field}. The \emph{value
group} of $(K,w)$ is $w(K):=G$. The \emph{valuation ring} of $w$
is $\vr{K} := \{a : a\in K \text{ and } w(a)\geq 0\}$ and
the \emph{valuation ideal} is $\I{K} := \{a : a\in K \text{ and
} w(a)> 0\}$. We denote by $\U{K}$ the multiplicative group $1 +
\I{K}$ (the group of 1-units); it is a subgroup of the group of
units (invertible elements) of $\vr{K}$. We denote by $P$ the
place associated to a valuation $w$; we denote the residue field
by $KP = \vr{K}/\I{K}$. (We shall omit the $K$ from the above
notations whenever it is clear from the context.) For $b \in
\vr{K}$, $bP$ or $bw$ is its image under the residue map.
For a subfield $E$ of $K$, we say that $P$ is \emph{$E$-rational}
if $P$ restricts to the identity on $E$ and $KP = E$.

\sn A valued field $(K,w)$ is \emph{henselian} if given
a polynomial $p(x) \in \mathcal{O}[x]$, and $a \in Kw$ a simple root of
the reduced polynomial $p(x)w\in Kw[x]$, we can find a root $b\in K$
of $p(x)$ such that $bw=a$.
\end{definition}

There are important examples of valued fields.
If $(K,+,\times,0,1,<)$ is an ordered field, we denote by $v$ its natural
valuation, that is, the \emph{natural valuation} $v$ on the ordered
abelian group
$(K,+,0,<)$. (The set of archimedean classes becomes an ordered abelian
group by
setting $[x]+[y]:=[xy]$.)
Note that the residue field in this case is an archimedean ordered
field, and that
$v$ is
\emph{compatible} with the order, that is, has a convex valuation ring.

Given an ordered abelian group $G$ and a field $F$,
denote by $F((G))$ the (generalized)
\emph{power series field} with coefficients in $F$
and exponents in $G$; elements of $F((G))$ take the form
$\sum_{g \in G} a_g t^g$
with $a_g \in F$ and well-ordered \emph{support}
$\{ g \in G : a_g \neq 0 \}$.
We define the \emph{minimal support valuation} on a non-zero
element $f \in F((G))$ to be $\vmin(f) = \min \support(f)$.
By convention, $\vmin(0) = \infty$.

\begin{definition}
Let $E$ be a field and $G$ an ordered abelian group.
Given $P$ a place on $E$, we define the ring homomorphism:
\[ \varphi_P : \vr{E}((G))
    \to (E P)((G)); \quad
\sum_g a_g t^g \mapsto \sum_g (a_g P) \, t^g \,. \]
\end{definition}

\subsection{Motivation and Results}
Brown in \cite{B} proved that a valued vector space of countable
dimension admits a valuation basis. This result was applied in
\cite{KS1} to show that every countable ordered field, henselian with
respect to its natural valuation,
admits a restricted exponential function, that is, an order
preserving isomorphism from the ideal of infinitesimals
$(\I{K}, +, 0)$ onto the group of 1-units $(\U{K}, \times,
1)$. We address the following question:
\emph{does every ordered field $K$, which is henselian with respect to its
natural valuation, admit a restricted exponential
function?} Let us consider the following illustrative example.

\begin{example}\textbf{Puiseux series fields:}
\label{psf example}
Let $F$ be a real closed field. Then the function
field $F(t)$ is an ordered field, where $0 < t < a$
for all $a \in F$. Define the real closed field of
(generalized) Puiseux series over $F$ to be
\[ \psf{F} = \bigcup_{n \in \mathbb{N}} F((t^{\frac{1}{n}})) \,, \]
and let $F(t)^\sim$ denote the real
closure of $F(t)$.
We then have the following containments
of ordered fields:
\[ F(t)
    \subset F(t)^\sim
    \subset \psf{F}
    \subset F((\mathbb{Q})) \,. \]

Now, since $F$ has characteristic $0$, then the
power series field $F((G))$ admits a $\vmin$-compatible
restricted exponential $\exp$ with inverse $\log$. These are defined by
\[ \exp(\varepsilon) = \sum_{i=0}^\infty \frac{\varepsilon^i}{i!}
    \quad \text{and} \quad
    \log(1 + \varepsilon) = \sum_{i=1}^\infty (-1)^{i+1}
    \frac{\varepsilon^i}{i}
    \quad \text{where} \quad
    \varepsilon \in \I{F} \,. \]
(See~\cite{A}.) The same argument as in the previous example shows
that each term
$F((t^{\frac{1}{n}}))$ in $\psf{F}$
admits a restricted exponential. Therefore, so does
$\psf{F}$ itself. We now turn to the question of whether
$F(t)^\sim$ admits a restricted exponential.
Note that one could not just take the restriction of the
exponential map $\exp$ defined above to the subfield
$F(t)^\sim \subseteq F((\mathbb{Q}))$.
Indeed, it can be shown that
the map $\exp$ sends algebraic power series
to transcendental power series,
so the restriction of the
exponential map $\exp$ to $F(t)^\sim$ is not even a well-defined map.
\end{example}
Following the strategy outlined at the beginning of this section,
we shall instead investigate whether the multiplicative
group of $1$-units and the valuation ideal of $F(t)^\sim$ admit valuation bases.
\mn
It turns out that this question is interesting to ask for any
valued field (not only for ordered valued fields):
\begin{definition}
Given a valued field $(L,w)$, define a
\emph{$w$-restricted exponential} $\exp$ to be an isomorphism
of groups between the valuation ideal of
$L$ and the $1$-units of $L$ (with respect to $w$)
which is \emph{$w$-compatible}; that is,
\[ wa = w(1 - \exp(a)) \,. \]
\end{definition}
The main results are Theorem~\ref{main-additive}
and Theorem~\ref{main-multiplicative} (see Section~\ref{main-results}).
We consider valued subfields $L$ of a field of power series $F((G))$,
where $F$ is algebraically (or real) closed, and G is a countable
ordered abelian group, which satisfy the
\emph{transcendence degree reduction property} (\textbf{TDRP}) over a
countable ground field $K$ (see Definitions~\ref{tdrp algebraic}
and \ref{tdrp real algebraic}; Section~\ref{main-results}).
We prove that the additive
group of $L$ admits a valuation basis as a $K$-valued vector space.
In particular, the valuation ideal of $L$
admits a valuation basis as a $K$-valued vector space.
If the group of 1-units of $L$ is divisible, we show that it admits a
valuation basis over the rationals. We exhibit
 some interesting intermediate fields
$F(G) \subseteq L \subseteq F((G))$ satisfying
the TDRP over $K$.
For instance, the field of rational
functions $F(G)$ and the field $\ac{F(G)}$ of power series
in $F((G))$ algebraic over $F(G)$ satisfy it (see Theorem~\ref{rat tdrp} and
Theorem~\ref{alg tdrp}).
We show that the class of fields satisfying the
TDRP over $K$ is closed under adjunction of countably
many elements of $K((G))$ --- if $L$ satisfies the
TDRP over $K$, then so does $L(f_1, f_2, \ldots)$
(see Theorem~\ref{tdrp adjoin}).
\mn
In particular, if $F$ is an archimedean ordered real closed field, and
$G$ is a countable divisible ordered abelian group, then the real closed
field $\ac{F(G)}$ admits a restricted exponential function. This gives a
partial answer to the original question posed.
\mn It is interesting to note that similar arguments are used in
Section~11, p.~35 of \cite{A-D} to show that certain ordered fields
admit a derivation function.
\bn
The paper is organized as follows. In Section~\ref{main-results},
we give a detailed statement of the main results.
In Section~\ref{tech-results}, we work
out several technical valuation theoretic results, needed for the proofs
of the main results. In Section~\ref{rat-alg-tests},
we develop interesting tests to
decide whether a generalized power series is rational, or algebraic over
the field of rational functions. In Section~\ref{the-tdrp},
we discuss the TDRP in
detail and prove Theorems~\ref{rat tdrp}, \ref{alg tdrp} and
\ref{tdrp adjoin}.
Section~\ref{unbounded-trdeg} is devoted to the proofs of
Theorems~\ref{main-additive} and \ref{main-multiplicative}.
Finally, in Section~\ref{applications} we apply the results to ordered fields
and to the complements of their valuation rings.

It turns out that by assuming $|F| \leq \aleph_1$,
one can provide elementary proofs of Theorems~\ref{main-additive} and \ref{main-multiplicative}
not requiring the technical machinery developed in Sections~\ref{tech-results} and
\ref{unbounded-trdeg}. We provide details in Appendix~\ref{Appendix}
(Theorems~\ref{main-ch-additive} and \ref{main-ch-multiplicative}).


\section{Main Results}
\label{main-results} In this paper, we will be particularly
interested in subfields of $F((G))$ satisfying a certain closure
property. We first provide a definition in the case where $F$ is
algebraically closed.

\begin{definition}[TDRP --- algebraic]
\label{tdrp algebraic}
Let $F$ be an algebraically closed field,
$K$ a countably infinite subfield of $F$ and $G$ a
countable ordered abelian group.
We say that an intermediate field $L$, for
\[ F(G) \subseteq L \subseteq F((G)) \,, \]
satisfies the \emph{transcendence degree reduction
property} (or TDRP) over $K$ if:
\begin{enumerate}
\item
whenever the intermediate field $E$, for $K \subseteq E \subseteq F$,
is countable, then $E((G)) \cap L$ is countable;
moreover, $L$ is the colimit\footnote{union over a directed set}
of the $E((G)) \cap L$ over such $E$;

\item
whenever $K \subseteq E \subset E' \subseteq F$ for
algebraically closed intermediate fields $E, E'$ and
$E'/E$ is a field extension of transcendence degree $1$,
then for finitely many power series $s_1, \ldots, s_n$
in $E'((G)) \cap L$, there exists an $E$-rational place
$P$ of $E'$ such that $s_i \in \vr{P}((G))$ and
$\varphi_P(s_i) \in E((G)) \cap L$ for all $i$;

\item
for $E, E', P$ as above,
if $\{ \alpha \}$ is a fixed transcendence basis of
$E'/E$, we may assume that $P$ sends $\alpha, \alpha^{-1}$
to $K$.
\end{enumerate}
\end{definition}
The key point of the third axiom is that if $P$ restricts
to the identity on some intermediate field $K \subseteq K' \subseteq E'$
and is finite on some element $c$ algebraic over $K'(\alpha)$,
then $cP$ is algebraic over $K'$.
(See the proof of Proposition~\ref{alg reduce trdeg}).
\mn

It turns out that many results for the real closed case
are implied by those for the algebraically closed case;
hence, we make the following analogous definition.
\begin{definition}[TDRP --- real algebraic]
\label{tdrp real algebraic}
Let $F$ a real closed field, $K$ a countably infinite subfield of $F$
and $G$ a countable ordered abelian group.
We say that an intermediate field $L$, for
\[ F(G) \subseteq L \subseteq F((G)) \]
satisfies the \emph{transcendence degree reduction
property} over $K$ if the intermediate field
\[ F^a(G) \subseteq (L \oplus \sqrt{-1} L) \subseteq F^a((G)) \]
does, where $F^a = F \oplus \sqrt{-1} F$ denotes the algebraic
closure of $F$.
\end{definition}

Note that an elementary argument from field theory shows
that $F^a(G) = F(G) \oplus \sqrt{-1} F(G)$; we give an
alternative argument in the proof of Theorem~\ref{rat tdrp}.

Consider an algebraically or real closed field $F$
and a countable ordered abelian group $G$.
We will exhibit later some interesting intermediate fields
$F(G) \subseteq L \subseteq F((G))$ satisfying
the TDRP over $K$.
For instance, the field of rational
functions $F(G)$ and the field $\ac{F(G)}$ of power series
in $F((G))$ algebraic over $F(G)$ satisfy it.
Moreover, the class of fields satisfying the TDRP
over $K$ is closed under adjunction of countably
many elements of $K((G))$ --- if $L$ satisfies the
TDRP over $K$, then so does $L(f_1, f_2, \ldots)$.

\begin{remark}
Note that $L(f_1, f_2, \ldots)$ doesn't necessarily
have countable dimension over $L$, so we cannot
resort to any generalization of Brown's theorem (\cite{B}) in
this situation.
\end{remark}

\mn
Our primary objective of this paper is to prove the
following result.
\begin{theorem}[Additive]
\label{main-additive}
Let $F$ be an algebraically or real closed field,
$K$ a countably infinite subfield of $F$ and $G$ a countable ordered abelian
group. If $F(G) \subseteq L \subseteq F((G))$ is an intermediate field
satisfying the TDRP over $K$, then the valued $K$-vector spaces
$(L, +)$ and therefore $(\I{L}, +)$ admit valuation bases.
\end{theorem}

We also prove the following multiplicative analogue.

\begin{theorem}[Multiplicative]
\label{main-multiplicative}
Let $F$ be an algebraically or real closed field
of characteristic zero, and $G$ a
countable ordered abelian group.
If $F(G) \subseteq L \subseteq F((G))$ is an intermediate field
satisfying the TDRP over $\mathbb{Q}$ and the group $(\U{L}, \times)$
is divisible, then $(\U{L}, \times)$ is a valued
$\mathbb{Q}$-vector space and admits a $\mathbb{Q}$-valuation basis.
\end{theorem}

Note that these results are trivial whenever $F$
is assumed to be countable; by the TDRP axioms, $L$
would be countable, and we could apply Brown's theorem (\cite{B}).
So, suppose $F$ is uncountable. Our strategy then
involves expressing uncountable objects, such as $F$,
as the colimits of countable objects.
In particular, suppose we express $F$ as
the colimit of countable subfields, say $K_\lambda$
for indices $\lambda$ in a directed set. (This is
always possible; how we do it will depend whether we
may assume $\operatorname{tr deg} F \leq \aleph_1$.)
From this, it will follow that, in the additive situation,
the group $\I{L}$ is the colimit of the
countable groups $\I{K_\lambda((G)) \cap L}$;
in the multiplicative situation, the group
$\U{L}$ is the colimit of the
countable groups $\U{K_\lambda((G)) \cap L}$.

We now restrict ourselves to the additive case;
analogous remarks apply to the multiplicative case.
Since each $\I{K_\lambda((G)) \cap L}$ is countable, we
can find a valuation basis for it by Brown's theorem (\cite{B}),
say $B_\lambda$.
If we are fortunate enough that these valuation bases
are consistent in the sense that $B_{\lambda'}$ extends
$B_\lambda$ whenever $\lambda \prec \lambda'$, then
we may take the colimit of the $B_\lambda$, which will
be our desired valuation basis of $\I{K_\lambda((G)) \cap L}$.
How are we to choose the $B_\lambda$ consistently?
The answer lies in a generalization of
Brown's theorem (\cite{B}), featured as Corollary~{3.6} in \cite{KS2}.

\begin{definition}
Let $V/W$ be an extension of valued $k$-vector spaces
with valuation $w$.
For $a \in V$, we say that $a$ has an
\emph{optimal approximation} in $W$ if there
exists $a' \in W$ such that for all $b \in W$,
$w(a' - a) \geq w(b - a)$.
We say that $W$ has the \emph{optimal
approximation property} in $V$ if every
$a \in V$ has an optimal approximation in $W$.
\end{definition}

The following proposition follows from Corollary~3.6 in \cite{KS2}.
(There, the term ``nice'' is used for
the optimal approximation property.)

\begin{proposition}
\label{brown}
Let $V/W$ be an extension of valued $k$-vector spaces.
If $W$ has the optimal approximation property in $V$
and $\dim_k V/W$ is countable, then any $k$-valuation
basis of $W$ may be extended to one of $V$.
\end{proposition}

We are then left show to show that $\I{K_\lambda((G)) \cap L}$
has the optimal approximation property in
$\I{K_{\lambda'}((G)) \cap L}$ whenever $\lambda \prec \lambda'$;
this will occupy the bulk of our arguments.
Once we establish this, we are able to easily construct
our desired valuation bases inductively.
\mn
We conclude with two remarks concerning the two main theorems.

\begin{remark}
Note that the assumption that $\ch{F} = 0$ is necessary in
Theorem~\ref{main-multiplicative}. If $\ch{F} = p$, then for
any non-trivial element $f \in \U{L}$, we have
$\vmin(1 - f^p) = p \cdot \vmin(1 - f) \neq \vmin(1 - f)$.
Hence, $(\U{L}, \times)$ does not admit a valued $\mathbb{Q}$-vector
space structure, even if it is divisible.
\end{remark}

\begin{remark}
Note that it can make a difference over which subfield we wish to
take a valuation basis.
By the results of this paper, we know that
$\mathbb{R}(t)$ and $\mathbb{R}(t)^\sim$ both admit
$\mathbb{Q}$-valuation bases.
We claim they do not admit $\mathbb{R}$-valuation bases.
Indeed, since $\mathbb{R}(t)$ and $\mathbb{R}(t)^\sim$
have residue field $\mathbb{R}$, if $\mathcal{B}$ is an
$\mathbb{R}$-valuation independent subset, then the elements
of $\mathcal{B}$ have pairwise distinct values.
Therefore, $|\mathcal{B}| \leq |\mathbb{Q}| = \aleph_0$.
On the other hand, the dimension of $\mathbb{R}(t)$,
as a vector space
over $\mathbb{R}$ is uncountable (e.g.~the subset $\{ (1-xt)^{-1} \}_{x
\in
\mathbb{R}}$ is $\mathbb{R}$-linearly independent).
\end{remark}
\mn
Concerning the choice of the ground field,
we also record the following observation (which is of independent
interest). The proof is straightforward, and we omit it.
\begin{proposition}
Let $V$ be a valued $K$-vector space and $k$ be a subfield of $K$.
If $B$ denotes a $K$-valuation basis of $V$ and $B'$ denotes a
$k$-vector space basis of $K$, then
$B \otimes B' = \{ b \otimes b' : b \in B, b' \in B' \}$
is a $k$-valuation basis of $V$.
\end{proposition}

\section{Technical results and key examples}
\label{tech-results}

We isolate here some results common to the proofs of
our main theorems; note that the proofs of these results
hold in every characteristic unless noted otherwise.
As an application, we then give examples of fields
satisfying the TDRP.

\subsection{Constructing places}

A basic tool in this paper will be the existence
of certain places; these will often be used to
decrease transcendence degree.

\begin{proposition}
\label{place exists}
Consider a tower of fields
\[ K \subseteq E \subseteq E' \]
where $K$ is infinite and $E'/E$ is an extension of
algebraically closed fields with transcendence
basis $\{ \alpha \}$.
Suppose $R$ is a subring of $E'$ that is finitely
generated over $E$. Then there exists an $E$-rational place
$P$ of $E'$ such that the elements $\alpha$ and $\alpha^{-1}$
are sent to $K$ and the place $P$ is finite on $R$.
\end{proposition}

\begin{proof}
We assume without loss of generality that
$\alpha, \alpha^{-1} \in R$; if not, simply adjoin them.
\comment{Additionally, we may assume that
$\operatorname{Quot} R$ is separable over $E(\alpha)$.
Otherwise, we may
take a positive integer $m$ such that the
$p^m$-th powers of $\operatorname{Quot} R$
are separable over $E(\alpha)$; 
since $E$ is algebraically closed (the only option in
the positive characteristic case) and
therefore perfect, it follows that the field
$\operatorname{Quot} R(\alpha^{1/p^m})$
is separable over $E(\alpha^{1/p^m})$.
Formally replacing $\alpha$ by
$\alpha^{1/p^m}$ and assuming the desired
result, we obtain a place $P$ satisfying our
stated conditions.
}
We first exhibit a place of $\operatorname{Quot} R$
satisfying the stated conditions.

There are infinitely many $E$-rational places $P$ of
$\operatorname{Quot} R$ sending $\alpha$ and $\alpha^{-1}$ to $K$.
Indeed, for each $q \in K$, we obtain the
$(\alpha - q)$-adic place $P_q$
on $E[\alpha]$ and therefore
on $\operatorname{Quot} R$ by Chevalley's place
extension theorem.
Note that for $q \neq q'$, we necessarily have
$P_q \neq P_{q'}$.

Moreover, we may select some $q$ such that
$P_q$ is finite on $R$.
For suppose $R = E[c_1, \ldots, c_n]$. Since the
$P_q$ are trivial on $E$, they are necessarily
finite on any $c_i$ algebraic over $E$.
On the other hand, for any $c_i$ transcendental
over $E$, the $(1/c_i)$-adic place on
$E(c_i)$ is the only one not finite on $c_i$; by
extension, there are at most
$[\operatorname{Quot}(R) : E(c_i)] < \infty$
places on $\operatorname{Quot}(R)$ not finite on
$c_i$. (Precisely how many depends on separability.)
Since of the infinitely many places $P_q$ only
finitely many map $c_i$ to $\infty$ for some $i$,
we may fix a $q$ such that $P_q$ is finite on
all $c_i$ and thus finite on $R$.

Henceforth, write $P$ to denote this place.
By Chevalley's place extension theorem again, $P$
on $\operatorname{Quot} R$ extends to a place on
$E'$ satisfying our desired properties.
\end{proof}

Intuitively, the place $P$ given by
Proposition~\ref{place exists} is used to replace a field
subextension of $K$ in $F$ of transcendence degree $d$
by one of transcendence $d - 1$.
We may also make use of this tool for power series via
the induced ring homomorphism $\varphi_P$.
We now present a finiteness condition that enables
us to apply this previous result.

\begin{definition}
Let $(L,w)$ be a valued field. A \emph{contraction}
$\Phi$ on a subset $S$ of $L$ is a map
$S \to S$ such that
\[ w(\Phi a - \Phi b) > w(a-b) \text{ for all } a,b \in S \,. \]
\end{definition}

\begin{proposition}
\label{fg subring exists}
For $K$ a field and $G$ an ordered abelian
group, let $f \in K((G))$.
Let $f$ be algebraic over $K(G)$. If $\ch(K) = 0$, then
there exists a subring $R \subseteq K$ finitely generated
over $\mathbb{Z}$ such that $\coeffs f \subseteq R$. If
$\ch(K) = p > 0$, then there exists a subring $R \subseteq K$
finitely generated over $\mathbb{F}_p$ such that
$\coeffs f \subseteq R^{1/p^\infty}$.
\end{proposition}

\begin{proof}
We prove a stronger result. Namely,
let $L$ be the algebraic closure of $K$,
$H$ be the divisible hull of $G$,
$v$ be the minimal support valuation $\vmin$.
\sn
If $\ch(K) = 0$, define
\begin{eqnarray*}
S & = & \{ f \in L((H)) : \coeffs f \subset R
    \text{ for a subring $R \subseteq K$} \\
& & \quad \quad \text{finitely generated over $\mathbb{Z}$} \} \,,
\end{eqnarray*}
and if $\ch(K) = p > 0$, define
\begin{eqnarray*}
S & = & \{f \in L((H)) : R^{1/p^\infty} \text{ contains } \coeffs(f)
\text{ for a subring } R \\
& & \quad \quad \text{finitely generated over } \mathbb{F}_p \} \,.
\end{eqnarray*}

We show that $S$ is an algebraically closed subfield of $L((H))$.
For notational convenience, define $A = \mathbb{Z}$ if $\ch(K) = 0$;
otherwise, define $A = \mathbb{F}_p$.
\mn
We first establish that $(S,v)$ is a henselian subfield.
It is easily verified that $S$ is in fact a field
--- for if $r,r' \in S$ are contained in finitely
generated subrings $R, R'$, respectively, then
$r - r'$ belongs to the finitely generated ring
$A[R,R']$; if $r' \neq 0$, then $r/r'$
belongs to the finitely generated ring
$A[R,R',1/c]$, where $c$ is the leading
coefficient of $f'$.

We now verify Hensel's Lemma.
Take a monic polynomial $Q \in \vr{S}[t]$ and an
approximate root $r \in \vr{S}$ such that $v Q(r) > 0$
and $v Q'(r) = 0$.
Write $Q(t) = a_0 + a_1 t + \cdots + a_n t^n$, and
let $c$ be the leading coefficient of $Q'(r)$.
We claim that $r$ can be refined to a root $f$
such that $\coeffs f \subseteq R$, where $R$
is the ring $A[1/c, \coeffs(a_i, r)]$.
By the Newton Approximation Method, we obtain a
contraction:
\begin{gather*}
\Phi : r + \I{R((G))} \to r + \I{R((G))} \\
    x \mapsto x - Q(x)/Q'(r) \,.
\end{gather*}
Since $\I{R((G))}$ is spherically complete, $\Phi$ has a fixed
point, which is a root of $Q$ in $r + \I{R((G))}$. Thus, $S$ is
henselian. \mn First assume that $\ch(K) = 0$. Since the value group
$v S = H$ is divisible and the residue field $S v$ is algebraically
closed, it follows that $S$ is algebraically closed. (See \cite{P}.) \sn
Now assume that $\ch(K) = p > 0$. The algebraic closure of $S$ is a
purely wild extension of $S$, and by Lemma~13.11 in \cite{KF2}, $S$
is equal to its own ramification field; by Theorem~7.15 in
\cite{KF2} (which states that the ramification group is a pro-$p$
group), if $S$ is not algebraically closed, then we can find a
separable extension $S'$ of $S$ of degree $p$. Such an extension
$S'$ is generated by an Artin-Schreier polynomial by Theorem 6.4 of \cite{L};
however, this is
impossible, since any root of an Artin-Schreier polynomial is once
again contained in $S$. \mn Since $S$ clearly contains $K(G)$, the
desired result then follows; namely, that whenever $f$ is algebraic
over $K(G)$, then $f \in S$.
\end{proof}

Note that in positive characteristic, the statement
that $\coeffs f \subseteq R^{1/p^\infty}$ cannot be
strengthened to $\coeffs f \subseteq R$. Indeed, let
$K = \mathbb{F}_p(y)$ and $G = \mathbb{Q}$.
Then the power series
\[ f(t) = \sum_{i \geq 1} y^{1/p^i} t^{-1/p^i} \]
satisfies the relation $f^p - f - yt^{-1}$ and
is therefore algebraic over $K(\mathbb{Q})$;
on the other hand, the coefficient set of $f(t)$
is $\{ y^{1/p^i} \}$, which is clearly not contained
in any ring finitely generated over $K = \mathbb{F}_p(y)$.
\mn
We now apply these results to rational and algebraic
series.

\begin{proposition}
\label{rat reduce trdeg}
Let $E'/E$ be an extension of algebraically closed fields
with transcendence basis $\{ \alpha \}$
and take an infinite subfield $K$ of $E$.
Given finitely many power series $s_1, \ldots, s_n \in E'(G) \subseteq E'((G))$,
there exists an $E$-rational place $P$ of $E'$
sending $\alpha, \alpha^{-1}$ to $K$
such that $s_i \in \vr{P}((G))$ and $\varphi_P(s_i) \in E(G) \subseteq E((G))$
for each $i$.
\end{proposition}

\begin{proof}
For each $i$, take $f_i, g_i \in E'[G]$ such that $s_i = f_i/g_i$;
without loss of generality, assume that the $g_i$ are monic.
Observe that $\coeffs(s_i, f_i, g_i)$ are contained in the
finitely generated ring $R[\coeffs(f_i, g_i)]$; hence, by
Proposition~\ref{place exists}, there exists an $E$-rational
place $P$ of $E'$ sending $\alpha, \alpha^{-1}$ to $K$
that is finite on $R$.
Since each $g_i$ is monic, the $\varphi_P(g_i)$ are non-zero;
hence,
$\varphi_P(s_i) = \varphi_P(f_i)/\varphi_P(g_i)$.
\end{proof}

\begin{proposition}
\label{alg reduce trdeg}
Let $E'/E$ be an extension of algebraically closed fields
with transcendence basis $\{ \alpha \}$
and take a infinite subfield $K$ of $E$.
Given finitely many power series $s_1, \ldots, s_n \in E'((G))$
that are algebraic over $E'(G)$, there exists an
$E$-rational place $P$ of $E'$ sending $\alpha, \alpha^{-1}$ to $K$
such that $s_i \in \vr{P}((G))$ and $\varphi_P(s_i)$ is algebraic
over $E(G)$ for each $i$.
\end{proposition}

\begin{proof}
By Proposition~\ref{fg subring exists}, there exists
a subring $R$ of $E'$, finitely generated over $E$,
such that $\coeffs s_i \subseteq R$ (if $\ch{E} = 0$)
or $\coeffs s_i \subseteq R^{1/p^\infty}$ (if $\ch{E} = p$)
for each $i$.
Pick a transcendence basis $\{ \alpha \}$ of $E'/E$.
Then, by Proposition~\ref{place exists}, we may
take an $E$-rational place $P$ of $E'$ that is
finite on $R$, $\alpha$ and $\alpha^{-1}$ and sends
$\alpha, \alpha^{-1}$ to $K$.

Take $s$ to be any of the $s_i$.
As $s$ is algebraic, suppose it is a root of the
non-trivial polynomial
$Q \in E[\alpha, t^g : g \in G][y]$.
Notice that in the polynomial ring $E[\alpha]$, the kernel
of $P$ is the prime ideal $(\alpha - \alpha P)$. Since
$E[\alpha]$ is a unique factorization domain, we may
divide out coefficients of $Q$ if necessary in order to
assume that the polynomial $\varphi_P Q$ is non-zero.
(In a slight abuse of notation, we extend $\varphi_P$
to the polynomial ring over $\vr{P}((G))$.)
As $\varphi_P s$ is a root of $\varphi_P Q \neq 0$,
it is algebraic over $E(G)$, as desired.
\end{proof}

\subsection{Coefficient tests for rational and algebraic power series}
\label{rat-alg-tests}
Using the results developed in the previous section,
we can develop a simple coefficient test; in this
section, $G$ will denote an arbitrary ordered
abelian group with no restrictions on its cardinality.
For now, we make no assumptions about characteristic.

\begin{proposition}
\label{rat test}
Let $E/K$ be an extension of fields. Then,
\[ K((G)) \cap E(G) = K(G) \,. \]
\end{proposition}

\begin{proof}
It suffices to show that
\[ K((G)) \cap E[G] = K[G] \,, \]
as the desired result follows by taking function fields of
both sides.
This is clear, since $f \in K((G)) \cap E[G]$
means that $f$ has finite support with coefficients in $K$.
\end{proof}

We have an algebraic power series analogue corresponding
to Proposition~\ref{rat test}.

\begin{proposition}
\label{alg test}
Let $E/K$ be an extension of fields. If $E$ and $K$ are both
real closed or both algebraically closed, then
\[ K((G)) \cap \ac{E(G)} = \ac{K(G)} \,, \]
where $\ac{\cdot}$ denotes relative algebraic closure in
$E((G))$.
\end{proposition}

\begin{proof}
Letting $H$ denote the divisible hull of $G$, we see that $E((H))$
is algebraically or real closed if $E$ is algebraically or real closed,
respectively. The inclusion ``$\supseteq$'' follows immediately.

To see the ``$\subseteq$'' inclusion, first
assume that $E,K$ are algebraically closed. Take a power series
$s \in K((G)) \cap \ac{E(G)}$; since $s$ satisfies a polynomial
relation in $E(G)$, we may assume that $\trdeg E/K$ is finite by
replacing $E$ by a subextension of $K$ if necessary.
Taking a filtration
\[ K = E_0 \subset E_1 \subset \cdots \subset E_n = E \]
where $\trdeg E_{i+1}/E_i = 1$ for all $i$, we apply
Proposition~\ref{alg reduce trdeg} $n$ times to see that
$s \in K((G)) \cap \ac{E(G)}$, as desired.
If $E,K$ are real closed, by reducing to the algebraically closed case,
it suffices to note that any element $f$ of $K((G))$ that is algebraic
over $K^a((G))$ is also algebraic over $K((G))$.
\end{proof}

\subsection{TDRP for rational and algebraic power series}
\label{the-tdrp}

Fix an algebraically or real closed field $F$, a countably infinite
subfield $K$ and a countable ordered abelian group $G$.
In this section, we exhibit some intermediate fields
$F(G) \subseteq L \subseteq F((G))$ satisfying the TDRP over
$K$.

\begin{theorem}
\label{rat tdrp}
Suppose that $L = F(G)$. Then, $L$ satisfies the TDRP over $K$.
\end{theorem}

\begin{proof}
Suppose that $F$ is real closed; we give a reduction to the
case when $F$ is algebraically closed. Indeed, note that
$F^a = F \oplus \sqrt{-1}F$,
and we may take $\sigma$ to be the non-trivial element of
$\Gal{F^a/F}$.
It is clear that $F(G) \oplus \sqrt{-1}F(G)$ is contained in
$F^a(G)$.
On the other hand, $F^a(G)$ is contained in
$F(G) \oplus \sqrt{-1}F(G)$; for if $h$ is the power series
development of a rational function in $F^a(G)$, then
$h = (h + \sigma(h))/2 + (h - \sigma(h))/2$ and these two
summands are in $F(G)$ and $\sqrt{-1}F(G)$, respectively, by
Proposition~\ref{rat test}.
Hence, $F(G) \oplus \sqrt{-1}F(G) = F^a(G)$, and it suffices
to prove the theorem when $F$ is algebraically closed by
definition of the TDRP.

Thus, assume that $F$ is algebraically closed. The first
condition of the TDRP is obvious --- if $E$ is a field
extension of $K$ and $E$ is countable, then
$E((G)) \cap L = E(G)$ (with equality from Proposition~\ref{rat test})
is countable.
The second and third conditions are simply the statement of
Proposition~\ref{rat reduce trdeg}.
\end{proof}

\begin{theorem}
\label{alg tdrp}
Suppose that $L = \ac{F(G)}$, the relative algebraic closure
of $F(G)$ in $F((G))$. Then, $L$ satisfies the TDRP over $K$.
\end{theorem}

\begin{proof}
As above, we may assume that $F$ is algebraically closed after
verifying that $\ac{F(G)} \oplus \sqrt{-1} \ac{F(G)} = \ac{F^a(G)}$,
where $\ac{\cdot}$ denotes relative algebraic closure in $F((G))$ for
the first two instances and in $F^a((G))$ for the third.
Note that any element $f$ of $K((G))$ that is algebraic
over $K^a((G))$ is also algebraic over $K((G))$.
Verification of the TDRP properties proceeds nearly identically;
for the second condition of the TDRP, use
Proposition~\ref{alg reduce trdeg} instead of \ref{rat reduce trdeg}.
\end{proof}

We now show that the class of fields satisfying the
TDRP over $K$ is closed under the adjunction of countably
many power series in $K((G))$.

\begin{lemma}
\label{tdrp adjoin lemma}
Suppose that the intermediate field $F(G) \subseteq L \subseteq F((G))$
satisfies the TDRP over $K$, where $F$ is algebraically closed.
Consider an algebraically closed and countable subextension
$K \subseteq E \subseteq F$.
Then, for any power series $h \in K((G))$, we have
\[ L(h) \cap E((G)) = (E((G)) \cap L)(h) \,. \]
\end{lemma}

\begin{proof}
We show the ``$\subseteq$'' direction; the other is clear.

Suppose that $s \in L(h) \cap E((G))$; we may take some
countable algebraically closed field $E'$ containing $E$
such that
\[ s = (f_0 + f_1 h + \cdots + f_n h^n)/(g_0 + g_1 h + \cdots + g_m h^m) \]
for $f_i, g_i \in L \cap E'((G))$.
If $h$ is algebraic over $L$, we may assume the denominator
above is $1$; otherwise, we may assume that $g_0 = 1$.
Without loss of generality, we may take a chain
$E = E_0 \subset E_1 \subset \cdots \subset E_n = E'$
of algebraically closed intermediate fields $E_i$ such that the
$E_{i+1}/E_i$ are extensions of transcendence degree $1$.
By applying the second property of the TDRP $n$ times
to the displayed equation above, the first statement follows;
note that our assumption on the denominator implies that it
does not vanish.
\end{proof}

\begin{theorem}
\label{tdrp adjoin}
Suppose that the intermediate field $F(G) \subseteq L \subseteq F((G))$
satisfies the TDRP over $K$. Then if $\{ h_i \}_{i \geq 1}$
are power series in $K((G))$, the field $L(h_i : i \geq 1)$
also satisfies the TDRP over $K$.
\end{theorem}

\begin{proof}
As usual, it suffices to prove the result when $F$ is
algebraically closed.
Indeed, suppose that $\{ h_i \}_{i \geq 1}$ are power series in $K((G))$.
Then, it is easily shown that
\[ L(h_i : i \geq 1) \oplus \sqrt{-1}L(h_i : i \geq 1)
    = (L \oplus \sqrt{-1}L)(h_i : i \geq 1) \,; \]
the definition of TDRP for the algebraically closed field $F^a$ then applies.

Henceforth, suppose $F$ is algebraically closed.
To simplify notation, we will denote $E((G)) \cap L$ by $L_E$ for
any field $E$.
It suffices to verify the second condition of the TDRP, the
rest being trivial.
Furthermore, it suffices to show that if $L$ satisfies the
TDRP over $K$, then so does $L(h)$ --- given finitely many
power series $s_1, \ldots, s_n$ in
\[ L(h_i : i \geq 1) \cap E((G)) = L_E(h_i : i \geq 1) \]
(with equality from Lemma~\ref{tdrp adjoin lemma}),
we may select finitely many $h_1, \ldots, h_m$ (after reindexing)
such that $s_1, \ldots, s_n \in L(h_1, \ldots, h_m)$.

We proceed with the proof. Let $E,E'$ be algebraically closed fields and
$E'/E$ an extension of transcendence degree $1$.
Given $s_1, \ldots, s_n$ in $E'((G)) \cap L(h)$, we may write
$s_i = S_i(h)/Q_i(h)$ for polynomials $S_i(x),Q_i(x)$ in $L[x]$ by
Lemma~\ref{tdrp adjoin lemma}.
Moreover, if $h$ is algebraic, we assume the $Q_i$ are constant;
otherwise, assume that each $Q_i$ is monic.

Since $L$ satisfies the TDRP over $K$, we may take an $E$-rational
place $P$ of $E'$ such that $\coeffs (S_i, Q_i) \subseteq \vr{P}((G))$
and $\varphi_P(\coeffs (S_i, Q_i) ) \subseteq L_E$.
(Observe that since $S_i, Q_i$ are considered polynomials in
$L[h]$, their coefficients lie in $L \subseteq F((G))$.)
Extending $\varphi_P$ to polynomials over $L_{E'}$, we see that
$\varphi_P(S_i), \varphi_P(Q_i)$ are polynomials over $L_E$;
hence, $\varphi_P(S_i)(h), \varphi_P(Q_i)(h)$ are elements oF
$L_E(h)$.
Recall that if $h$ is algebraic over $L$, then the $Q_i$ are constant;
otherwise, they are monic. Hence, the $\varphi_P(Q_i)(h)$ are non-zero
and therefore $\varphi_P(s_i) \in L_E$ for all $i$, as desired.
\end{proof}

\begin{remark}
Since $G$ is countable, there exists a countable extension
field of $K$ containing the coefficients of countably
many power series in $K((G))$. In particular, this means that
if $L \subseteq F((G))$ satisfies the TDRP over $K$, then for
countably many power series $(h_i)_{i \geq 1}$ in $F((G))$, there exists
a countable extension field $K'/K$ such that
$L(h_i : i \geq 1)$ satisfies the TDRP over $K'$.
\end{remark}

\section{Constructing valuation bases via TDRP}
\label{unbounded-trdeg}

In this section, we seek out to prove Theorems~\ref{main-additive}
and \ref{main-multiplicative}.
In what follows, $F$ denotes an algebraically or real closed field,
and we consider a countable subfield $K \subset F$.

Our strategy is to express $F$ as the colimit of countable subfields
of finite transcendence degree over $K$.
More precisely, fix a transcendence basis
$\{ \alpha_\lambda \}_{\lambda \in I}$ of $F$ over $K$.
Notice that the family of finite subsets of $I$ forms a directed
set under inclusion --- for each such finite subset $X \subset I$,
define the subfield
\[ K_X = \ac{K(\alpha_\lambda : \lambda \in X)} \subseteq F \,, \]
where $\ac{\cdot}$ denotes relative algebraic closure in $F$.
Observe that just as $\varinjlim X = I$, $\varinjlim K_X = F$.
Moreover, by the first TDRP axiom,
\begin{eqnarray*}
\varinjlim K_X((G)) \cap L
    & = & L \,, \\
\varinjlim \I{K_X((G)) \cap L}
    & = & \I{L} \quad \text{and} \\
\varinjlim \U{K_X((G)) \cap L}
    & = & \U{L} \,.
\end{eqnarray*}

Given any finite subset $X$ of $I$, we will need the optimal
approximation property for the valued vector space extensions
\[ \spn{ \I{K_Y((G)) \cap L} : Y \subset X }
    \subseteq \I{K_X((G)) \cap L} \,. \]
Consequently, we will fix $X$ throughout this section.
For notational convenience, label the elements
of $X$ to be $x_1, x_2, \ldots, x_N$, so that
\[ X = \{ x_1, x_2, \ldots, x_N \} \,; \]
for $1 \leq i \leq N$, let $Y_i = X \setminus \{ x_i \}$ and
$Y_{i,j} = X \setminus \{ x_i, x_j \}$.

Our desired results in the case that
$F$ is real closed will follow from the corresponding results when
$F$ is algebraically closed.
Hence, we will assume that $F$ is algebraically closed for now.

\subsection{Complements of valuation rings in characteristic $0$}

The relevance of this section to the rest of the paper is to
establish Theorem~\ref{place output} in the sequel; the second half of this
section is technically unnecessary and is provided
for the sake of independent interest and perspective.

Out of necessity, $\ch{F} = 0$ throughout.
For simplicity, we assume also that $F$ is algebraically closed.

Suppose that we have a $K_{Y_N}$-rational place $P$ of $K_X$
sending $\alpha_N, \alpha_N^{-1}$ to $K$.
Consider a sum $a$ of elements of the $K_Y$ for $Y \subset X$;
that is,
\[ a = a_1 + a_2 + \cdots + a_N \text{ for } a_i \in K_{Y_i} \,. \]
We would like to show that whenever $aP$ is finite,
we may assume that we also have a representation of the form
\[ a = b_1 + b_2 + \cdots + b_N \text{ for } b_i \in K_{Y_i} \,, \]
where each $b_i P$ is finite.

Note that $\vr{K_X}$ is a $\res{K_X}$-vector space, where
$\res{K_X}$ denotes the residue field of $K_X$, so
there exists a $\res{K_X}$-vector space complement
$C$ of $\vr{K_X}$ in $K_X$; that is, $K_X = C \oplus \vr{K_X}$.
Observe that for each $1 \leq i \leq k$,
\[ (K_{Y_i} \cap C) \oplus (K_{Y_i} \cap \vr{K_X})
    \subseteq K_{Y_i} \,. \]
Assuming equality held in the equation above,
we could uniquely write $a_i = b_i + c_i$
for $b_i \in \vr{K_{Y_i}}$ and $c_i \in C$ ---
note that $\vr{K_{Y_i}} = \vr{K_X} \cap K_{Y_i}$.
Our immediate aim is therefore to construct such a
complement $C$ where such equality in fact holds.

\begin{lemma}
\label{complement-ac}
Suppose that $F$ is algebraically closed and $P$ is a
$K_{Y_N}$-rational place of $K_X$ sending $\alpha_N, \alpha_N^{-1}$
to $K$.
Then, there exists a complement $C$ of $\vr{K_X}$ in $K_X$
such that for each $1 \leq i \leq N$,
$C \cap K_{Y_i}$
is a complement of
$\vr{K_{Y_i}} = \vr{K_X} \cap K_{Y_i}$ in $K_{Y_i}$.
\end{lemma}

\begin{proof}
Let $\psf{\res{K_X}}$ denote the field of Puiseux
series over $\res{K_X}$; that is,
\[ \psf{\res{K_X}} = \bigcup_{n=1}^\infty
    \res{K_X} ((t^{1/n})) \,. \]
We consider $\psf{\res{K_X}}$ to be a valued field
with the minimal support valuation $\vmin$.
Since the residue field $\res{K_X}$ is algebraically
closed and of characteristic $0$, it is well-known
that $\psf{\res{K_X}}$ is algebraically closed.

Note that since we consider $F$ to be
algebraically closed, we have that
$\res{K_X} = K_{Y_N}$.
As $\alpha_{x_N} P, \alpha_{x_N}^{-1} P \in K$
by construction, we see that the element
$\beta = \alpha_{x_N} - \alpha_{x_N} P \in K_{\{x_N\}}$
is transcendental over $\res{K_X} = K_{Y_N}$;
note that $\beta P = 0$.
We thus define the embedding
\[ \iota : K_{Y_N}(\beta) \to \psf{\res{K_X}} \]
such that $\iota$ restricts to the identity on
$K_{Y_N}$ and sends $\beta$ to $t$.
Since $\beta P = 0$, we have that $\iota$
preserves the valuation $v_P$ on
$K_{Y_N}[\beta]$; it follows that it does so
on $K_{Y_N}(\beta)$ as well.
Another easy consequence of $\beta P = 0$ is that
$\vr{K_{Y_i}} = K_{Y_{i,N}}$; this is proved as was
Proposition~\ref{alg reduce trdeg}.

Since $\psf{\res{K_X}}$ is algebraically closed
and $K_X$ is an algebraic field extension of
$K_{Y_N}(\beta)$,
$\iota$ extends to an embedding:
\[ \iota : K_X \to \psf{\res{K_X}} \,. \]
Note that this induces a valuation $w = \vmin \circ \iota$
on $K_X$. We may assume without loss of generality
that $w = v_P$; for $K_X$ is algebraic over $K_{Y_N}(\beta)$,
and therefore there exists $\sigma \in \Gal{K_X / K_{Y_N}(\beta)}$
such that $w \circ \sigma = v_P$. Thus, if we consider instead
the embedding $\iota' = \iota \circ \sigma$, we have that
$\iota'$ preserves valuations; that is,
$v_P = \vmin \circ \iota'$.

Moreover, for each $1 \leq i \leq N$, we have that
\[ \iota(K_{Y_i}) \subseteq \psf{\res{K_{Y_i}}} \,. \]
Note that this is immediate for $i = N$, as $\iota$
restricts to the identity on $K_{Y_N}$.
For $i \neq N$, notice that $K_{Y_i}$ is algebraic
over $K_{Y_{i,N}}(\alpha_{x_N})$;
moreover, $\iota$ restricted to
$K_{Y_{i,N}} \subset K_{Y_N}$ is the identity.
Consequently, $\iota(K_{Y_i})$ is algebraic over
$K_{Y_{i,N}}(\iota(\alpha_{x_N}))$.
Since $\alpha_{x_N} = \beta + \alpha_{x_N} P$,
we have
$\iota(\alpha_{x_N}) = t + \alpha_{x_N} P$;
this implies that
$\iota(K_{Y_{i,N}}(\alpha_{x_N}))$
and, by algebraicity, $\iota(K_{Y_i})$ are contained in
$\psf{\res{K_{Y_i}}}$.

We are ready to construct our complement of $C$ as stated.
In the following displays, $S$ will denote a finite
negative subset of $\mathbb{Q}$.
Note first that
\[ C_P = \{ \sum_{q \in S} c_q t^q :
    c_q \in \res{K_X}
    \text{ and $S$ is a finite negative subset of
    $\mathbb{Q}$} \} \]
is a complement to the valuation ring of
$\psf{\res{K_X}}$. Moreover, since it is contained in the
image of $\iota$ and $\iota$ preserves value, we deduce
that $\iota^{-1}(C_P)$ is a complement
of $\vr{K_X}$. That is, if
\[ C = \iota^{-1}(C_P) = \{ \sum_{q \in S} c_q \beta^q :
        c_q \in \res{K_X} \text{ for some $S$} \} \,, \]
then
\[ K_X = C \oplus \vr{K_X} \,. \]

It remains to verify that for each $1 \leq i \leq N$,
\[ K_{Y_i} \subseteq
    (K_{Y_i} \cap C) \oplus (K_{Y_i} \cap \vr{K_X}) \,. \]
Note that for $i = N$, this follows immediately, for
$K_{Y_N} \subseteq \vr{K_X}$.
For other $i$, the fact that
$\iota(K_{Y_i}) \subseteq \psf{\res{K_{Y_i}}}$
(and $\vr{K_{Y_i}} = K_{Y_{i,N}}$) shows that
\[ \iota(K_{Y_i}) \cap C_P = \{ \sum_{q \in S} c_q t^q :
    c_q \in \res{K_{Y_i}}
    \text{ for some $S$} \} \,, \]
which is a complement of the valuation ring
$\vr{\iota(K_{Y_i})}$ in $\iota(K_{Y_i})$.
Pulling back by the value-preserving embedding $\iota$,
it follows that
\[ \iota^{-1}(\iota(K_{Y_i}) \cap C_P)
    = K_{Y_i} \cap C
    = \{ \sum_{q \in S} c_q \beta^q :
    c_q \in K_{Y_{i,N}} \} \]
is a complement to $\vr{K_{Y_i}}$ in $K_{Y_i}$, where
$Y_{i,N} = Y_i \setminus \{ x_N \}$; that is,
\[ K_{Y_i} = (K_{Y_i} \cap C) \oplus \vr{K_{Y_i}}
    = (K_{Y_i} \cap C) \oplus (K_{Y_i} \cap \vr{K_X}) \,. \]
\end{proof}

It is still possible to prove the previous result in the case
that $F$ is real closed; however, significantly more work is needed
to eliminate negative parts of the power series
given by $\iota$ in the proof above.
We do not provide details here, as it suffices to consider
the case that $F$ is algebraically closed for now.

We can now construct complements as suggested from
the beginning of this section.

\begin{lemma}
\label{finite coeff}
Let $\gspan{Y}{X}{K_Y}$
denote the additive subgroup of $K_X$
generated by the subgroups $K_Y$ and suppose that
$P$ is a $K_{Y_N}$-rational place of $K_X$ sending
$\alpha_{x_N}, \alpha_{x_N}^{-1}$ to $K$. Then, with
respect to the place $P$,
\[ \vr{\gspan{Y}{X}{K_Y}}
    = \gspan{Y}{X}{\vr{K_Y}} \,. \]
More precisely,
\[ \vr{\spn{K_Y : \{ N \} \subseteq Y \subset X}}
    = \gspan{Y}{X}{\vr{K_Y}} \,. \]
\end{lemma}

\begin{proof}
It suffices to show the ``$\subseteq$'' direction;
the other is immediate.
Suppose that $a \in \vr{\gspan{Y}{X}{K_Y}}$.
We may write $a = a_1 + a_2 + \cdots + a_N$ where
$a_i \in K_{Y_i}$ for proper subsets
$Y_i \subset X$. (Recall that $Y_i$ was defined to
be $X \setminus \{ x_i \}$, where
$X = \{ x_1, x_2, \ldots, x_N \}$.)

By Lemma~\ref{complement-ac}, we may take a decomposition
$K_X = C \oplus \vr{K_X}$ such that
for each $1 \leq i \leq N$,
\[ K_{Y_i} =
    (K_{Y_i} \cap C) \oplus (K_{Y_i} \cap \vr{K_X}) \,. \]
Accordingly, we thus write the $a_i$ as $b_i + c_i$,
for $b_i \in \vr{K_{Y_i}}$ and $c_i \in C$ ---
note that $\vr{K_{Y_i}} = \vr{K_X} \cap K_{Y_i}$.
Since $a = \sum b_i + \sum c_i$ is in $\vr{K_X}$,
it follows that $\sum c_i = 0$; that is,
$a = b_1 + b_2 + \cdots + b_k$.
Noting that $c_N = 0$ and therefore $a_N = b_N$, both claims
now follow.
\end{proof}

Note that in positive characteristic, we can
no longer assume that there exist complements as given in
Lemma~\ref{complement-ac}; the proof fails as we
can no longer assume that the negative part of the
support of an algebraic power series,
and particularly of an element in the image of $\iota$,
is finite. (For example, see the remarks following
Proposition~\ref{fg subring exists}.) See Theorem~\ref{place output}
in the next section for a weakened version of Lemma~\ref{finite coeff}
that holds independently of $\ch{F}$.

\subsection{Output of places}
Our later combinatorial arguments will depend on a cancellation
property of a ring homomorphism $\varphi_P$ implied by the result
here.

As before, we assume $K$ is countably infinite; this is
in order that Proposition~\ref{place exists} holds.
The following lemma is a weak analogue to
Lemma~\ref{finite coeff} that holds independently of
$\ch{F}$; if $\ch{F} = 0$, it follows as an immediate corollary.

Note that since $F$ is assumed to be algebraically closed,
$K_X P = K_{Y_N}$.

\begin{theorem}
\label{place output}
Let $P$ be a $K_{Y_N}$-rational place of $K_X$ sending
$\alpha_{x_N}, \alpha_{x_N}^{-1}$ to $K$.
Suppose that $a \in \vr{\gspan{Y}{X}{K_Y}}$; that is,
$a = a_1 + a_2 + \cdots + a_k$ for $a_i \in K_{Y_i}$
and $P$ is finite on $a$.
Then $aP \in \gspan{Y}{X}{K_{Y}}$;
that is, $aP = b_1 + b_2 + \cdots + b_k$ for
$b_i \in K_{Y_i}$.

Moreover, we may assume that $b_N = a_N$ and that for any
$i$ such that $a_i = 0$, then $b_i = 0$.
\end{theorem}

\begin{proof}
Take a field embedding $\iota$ as in the proof of
Lemma~\ref{complement-ac}, and define $b_i$ to be the
constant term of $\iota(a_i)$; that is,
$b_i = 0(\iota(f))$.
\end{proof}

\subsection{The optimal approximation property}
Consider an intermediate field $F(G) \subseteq L \subseteq F((G))$
satisfying the TDRP over $K$. We give a combinatorial formula
for an optimal approximation $h$ to a power series $f \in L$
in terms of ring homomorphisms $\varphi_P$ given by
the second axiom of the TDRP. Since it will follow that $s \in L$
as well, this conceptually means that the field $L$ is
``closed under taking optimal approximations.''

\begin{theorem}
\label{optimal:add}
Suppose $F(G) \subseteq L \subseteq F((G))$ is an
intermediate field satisfying the TDRP over $K$ and $F$ is
algebraically closed.
Take $\left< \cdot \right>$ in the context of additive
groups. If $f \in K_X((G)) \cap L$, then there exists for each
$1 \leq i \leq N$ a $K_{Y_i}$-rational place $P_i$ of $K_X$
such that
\[ h = f - ({\id} - \varphi_{P_1}) \circ \cdots \circ
    ({\id} - \varphi_{P_N}) f \]
is an element of $\gspan{Y}{X}{K_Y((G)) \cap L}$ and an optimal
approximation to $f$ in $\gspan{Y}{X}{K_Y((G))}$;
the respective statement holds for $\I{K_X((G))}$.
\end{theorem}

\begin{proof}
It suffices to prove the first statement.

We define the places $P_k$ by decreasing induction
on $k$ from $N$ to $1$. For notational ease, whenever the places
$P_i$ have been defined for all $k < i \leq N$, we define
\[ f_k = ({\id} - \varphi_{P_{k+1}}) \circ \cdots \circ
    ({\id} - \varphi_{P_N}) \, f \,; \]
moreover, for any $(N - k + 1)$-tuple
$\sigma = (e_k, e_{k+1}, \ldots, e_N)$ over $\mathbb{F}_2$,
we define
\[ f_\sigma = \psi^\sigma(f)\,, \quad \text{where }
    \psi^\sigma = (-\varphi_{P_k})^{e_k} \circ \cdots \circ
    (-\varphi_{P_N})^{e_N} \,. \]
(For any function $\varphi$, we let $\varphi^1 = \varphi$
and $\varphi^0 = {\id}$.)
Observe that this means $f_N = f = f_{()}$, which is in $K_X((G)) \cap L$.

Suppose that for some $k$, $P_i$ has been defined for
all $k < i \leq N$ and that $f_k$ is a power series in
$K_X((G)) \cap L$. Then by the second TDRP axiom, we may
take a $K_{Y_k}$-rational place $P_k$ of $K_X$ such that for all
$(N - k + 1)$-tuples $\sigma$ over $\mathbb{F}_2$,
$P$ is finite on $\coeffs(f_k, f_\sigma)$ and $\varphi_P(f_k)$,
$\varphi_P(f_\sigma)$ are power series in $K_{Y_k}((G)) \cap L$.
Note that we then have $f_{k-1} \in K_X((G)) \cap L$, as desired.

Having defined our places, we check our two properties hold.
Let $\sigma = (e_1, \ldots, e_N)$ denote a non-zero tuple. If
$i$ denotes the least index such that $e_i = 1$, then
$f_\sigma \in K_{Y_i}((G)) \cap L$ by the third TDRP axiom. Since
\[ -h = \sum f_\sigma \,, \]
the sum over non-zero $N$-tuples $\sigma$, we see that
$h \in \gspan{Y}{X}{K_Y((G)) \cap L}$, as desired.

To see that $h$ is an optimal approximation to $f$ in
$\gspan{Y}{X}{K_Y((G))}$, it suffices to show that if
$\alpha(f) \in \gspan{Y}{X}{K_Y}$ for some exponent $\alpha$,
then $\alpha(h) = \alpha(f)$. Indeed, for such an $\alpha$,
write
\[ \alpha(f) = a_1 + a_2 + \cdots + a_N
    \quad \text{where } a_i \in K_{Y_i} \,; \]
by decreasing induction on $k$ using Theorem~\ref{place output},
we may write
\[ \alpha(f_k) = b_1 + b_2 + \cdots + b_k
    \quad \text{where } a_i \in K_{Y_i} \,. \]
Hence, $\alpha(h) = \alpha(f - f_0) = \alpha(f)$, as desired.
\end{proof}

We have a multiplicative analogue.

\begin{theorem}
\label{optimal:mult}
Suppose $F(G) \subseteq L \subseteq F((G))$ is an
intermediate field satisfying the TDRP over $K$ and $F$ is
algebraically closed and of characteristic $0$.
Take $\left< \cdot \right>$ in the context of multiplicative
groups. If $f \in \U{K_X((G)) \cap L}$, then there exists for each
$1 \leq i \leq k$ a $K_{Y_i}$-rational place $P_i$ of $K_X$
such that
\[ h = f {\,/\,} ({\id} {\,/\,} \varphi_{P_1}) \circ \cdots \circ
        ({\id} {\,/\,} \varphi_{P_k}) \, f \]
is an element of $\gspan{Y}{X}{\U{K_Y((G)) \cap L}}$ and an optimal
approximation to $f$ in $\gspan{Y}{X}{\U{K_Y ((G))}}$.
\end{theorem}

\begin{proof}
The construction of the places $P_k$ proceeds identically
as in Theorem~\ref{optimal:add}.
Verification of two stated properties is a straightforward modification
from before, after recalling that the map
$\exp$ introduced in Example~\ref{psf example} is a group isomorphism
from $(\I{F((G))}, +)$ to $(\U{F((G))}, \times)$ with inverse
$\log$ such that $\vmin(1 - \exp(f)) = \vmin(f)$.
Moreover, note that the maps $\exp$ and $\log$ commute
with the ring homomorphism $\varphi_P$.

In particular, to check that each $f_\sigma$, for $\sigma$
a non-zero tuple, is contained in some $K_Y((G))$, simply note
that
\[
f_\sigma = \psi^\sigma(f)
    = (\exp \circ \psi^\sigma \circ \log)(f) \,.
\]
Similarly, $h$ is an optimal approximation to $f$ in
$\gspan{Y}{X}{\U{K_Y ((G))}}$ if and only if
$\log(h)$ is an optimal one to $\log(f)$ in
$\gspan{Y}{X}{\I{K_Y ((G))}}$.
\end{proof}

Our optimal approximation results that we will use to extend
valuation bases now follow immediately; note that we now also
consider the case when $F$ is real closed.

\begin{theorem}
\label{oap:add}
Suppose $F(G) \subseteq L \subseteq F((G))$ is an
intermediate field satisfying the TDRP over $K$ and $F$ is
a real closed or algebraically closed field.
Take $\spn{\cdot}$ in the context of additive groups.
Then,
\[ \gspan{Y}{X}{\I{K_Y((G)) \cap L}} \]
has the optimal approximation property in
$\I{K_X((G)) \cap L}$. 
\end{theorem}

\begin{proof}
If $F$ is algebraically closed, then this is deduced
immediately from Theorem~\ref{optimal:add}.
Otherwise, if $F$ is real closed, we reduce to the
algebraically closed case. In particular, given an
element $f \in \I{K_X((G)) \cap L}$, we may regard
$f$ as an element of $(L \oplus \sqrt{-1} L) \cap K_X^a ((G))$.
By definition, $(L \oplus \sqrt{-1} L)$ is a subfield of
the algebraically closed field $F^a$ and satisfies the TDRP
over $K$; applying the desired result, we obtain an optimal
approximation $s$ to $f$ in
\[ \gspan{Y}{X}{(L \oplus \sqrt{-1} L) \cap K_Y^a((G))} \,. \]
Taking $\sigma$ to be the non-trivial element of
$\Gal{(L \oplus \sqrt{-1} L)/L}$, sending $a + \sqrt{-1} b$ to $a - \sqrt{-1} b$
for $a, b \in L$, we see that $(s + \sigma(s))/2$ is an optimal approximation
to $f$ in $\gspan{Y}{X}{K_Y((G)) \cap L}$, as desired.
\end{proof}

\begin{theorem}
\label{oap:mult}
Suppose $F(G) \subseteq L \subseteq F((G))$ is an
intermediate field satisfying the TDRP over $K$ and $F$ is a
real closed or algebraically closed field of characteristic
$0$.
Take $\spn{\cdot}$ in the context of multiplicative groups.
Then,
\[ \gspan{Y}{X}{\U{K_Y((G)) \cap L}} \]
has the optimal approximation property in
$\U{K_X((G)) \cap L}$.
\end{theorem}

\begin{proof}
As above.
\end{proof}

\subsection{Extending valuation bases}

Using our optimal approximation results, we can now
exhibit valuation bases for $\I{L, +}$ and $\U{L, \times}$,
where $L$ is a subfield of $F$ satisfying the TDRP over $K$.
(Note that when $\ch{K} > 0$, only the additive case applies.)
Recall that we have chosen a transcendence basis
$\{ \alpha_\lambda \}_{\lambda \in I}$ of $F$ over $K$,
and for each finite subset $X$ of $I$, we have
$K_X = \ac{K(\alpha_\lambda : \lambda \in X)}$.

For each $X$, let $V_X$ denote the valued $K$-vector space
$\I{K_X((G)) \cap L}$. If $\U{K_X((G)) \cap L}$ is a
divisible group, let $W_X$ denote the valued
$\mathbb{Q}$-vector space $\U{K_X((G)) \cap L}$.

For successively larger $n$, our aim is to define a
valuation basis $B_X$ for each valued vector space $V_X$
(or $W_X$ in the multiplicative case) where $|X| = n$,
extending the valuation bases $B_X$ for $|X| < n$.
We first give a lemma assuring that the valuation bases
$B_X$ for $|X| = n$ can be chosen independently, as long
as they extend the valuation bases $B_Y$ for $Y \subset X$.

\begin{lemma}
\label{independent}
Let $\left< \cdot \right>$ denote $K$-vector
space span.
For a finite subset $X \subseteq I$,
\[ K_X \cap \spn{K_Z : Z \nsupseteq X \text{ finite}}
    = \gspan{Y}{X}{K_Y} \,. \]
\end{lemma}

\begin{proof}
We first assume that $F$ is algebraically closed.
Let $Z_1, \ldots, Z_k$ be finite subsets of the index set $I$
not containing the subset $X$, and suppose that
$y = y_1 + \cdots + y_k \in K_X$.
It then suffices to show $y \in \gspan{Y}{X}{K_Y}$.

Writing $Z = X \cup Z_1 \cup \cdots \cup Z_k$,
we see that $K_Z$ has finite transcendence
degree over $K_X$.
Hence, we may take a chain of algebraically closed fields
\[ K_X = E_0 \subset E_1 \subset \cdots \subset E_n = K_Z \]
where each field extension $E_{i+1}/E_i$ is of transcendence
degree $1$.
By repeated application of Proposition~\ref{place exists},
for decreasing values of $i$ from $n-1$ to $0$, we can take
an $E_i$-rational place $P_i$ of $E_{i+1}$ that is finite on the
$y_j P_{i+1} P_{i+2} \cdots P_{n-1}$ and sends the transcendence
basis of $E_{i+1}$ over $E_i$ to $K$.

By repeated application of Proposition~\ref{alg reduce trdeg},
$y_j P \in K_{X \cap Z_j}$ for all $1 \leq j \leq k$,
where we write $P$ to denote the composition of places
$P_0 P_1 \cdots P_{n-1}$. Since $X \cap Z_j$ is
a proper subset of $X$, we have
\begin{eqnarray*}
y & = & y P \\
    & = & y_1 P + \cdots + y_k P \\
    & \in & \left< K_{X \cap Z_1}, \ldots,
    K_{X \cap Z_k} \right> \\
    & \subseteq & \gspan{Y}{X}{K_Y} \,.
\end{eqnarray*}

Now if $F$ is real closed, then taking algebraic closures and
applying the result in the algebraically closed case, we see that
\[ K_X^a \cap \spn{K_Z^a : Z \nsupseteq X \text{ finite}}
    = \gspan{Y}{X}{K_Y^a} \,, \]
from which the desired result follows immediately.
\end{proof}

\begin{theorem}
\label{valbasis:add}
Let $n \geq 0$, and suppose that for each subset
$X$ of $I$ of cardinality at most $n$, we have a valuation
basis $B_X$ of
\[ V_X = \I{K_X((G)) \cap L} \,. \]
Suppose that
\[ \mathcal{B}_n = \bigcup (B_X : X \subseteq I, |X| \leq n) \]
is valuation independent.
Then, for each subset $X$ of $I$ of cardinality $n+1$, we
may define a valuation basis $B_X$ of $V_X$
such that
\[ \mathcal{B}_{n+1} = \bigcup (B_X : X \subseteq I, |X| \leq n+1) \]
is valuation independent.
\end{theorem}

\begin{proof}
Observe that since $\mathcal{B}_n$ is valuation independent,
$B$ must be inclusion-preserving.
Indeed, suppose $X' \subset X$ of cardinality at most $n$.
By assumption, $B_{X'} \cup B_{X}$ is valuation independent
and therefore a valuation basis of $V_X$. Since $B_X$ is
a valuation basis of $V_X$ and therefore maximally valuation
independent, we must have $B_{X'} \cup B_{X} = B_X$ and
$B_{X'} \subset B_X$.

We now define a valuation basis $B_X$ of $V_X$
for each subset $X$ of $I$ of cardinality $n+1$.
For such a subset $X$, observe that
$\bigcup (B_Y : Y \subset X)$ is a valuation basis
of $\gspan{Y}{X}{V_Y}$.
By Theorem~\ref{oap:add}, the subspace
$\gspan{Y}{X}{V_Y}$ has the optimal
approximation property in $V_X$;
moreover, since $V_X$ is countable,
it has countable dimension over
$\gspan{Y}{X}{V_Y}$.
Therefore, Proposition~\ref{brown}
allows us to extend $\bigcup (B_Y : Y \subset X)$ to a valuation
basis of $V_X$, and we define $B_X$ to be
this.

It remains to show that $\mathcal{B}_{n+1}$ is valuation
independent. Consider a finite sum
\[ a = c_1 b_1 + c_2 b_2 + \cdots + c_k b_k \]
for non-zero scalars $c_i \in K$ and distinct
elements $b_i \in \mathcal{B}_{n+1}$
such that $q = \vmin(v_1) = \vmin(v_2) = \cdots = \vmin(v_k)$.
Since we know $\mathcal{B}_n$ is valuation independent, we
may assume with loss of generality (by reindexing if necessary)
that there exists some subset $X \subset I$ of cardinality $(n+1)$
and an index $1 \leq j \leq k$ such that
\[ b_i \in B_X \setminus \bigcup (B_Y : Y \subset X) \]
if and only if $i \leq j$.
Since $B_X$ is valuation independent, the coefficient
$q(c_1 v_1 + c_2 v_2 + \cdots + c_j v_j)$ is in
$K_X \setminus \gspan{Y}{X}{K_Y}$. As the coefficient
$q(c_{j+1} v_{j+1} + c_{j+2} v_{j+2} + \cdots + c_k v_k)$
is clearly in $\spn{K_Z : Z \neq X, |Z| \leq n+1}$,
Lemma~\ref{independent} implies that $q(a) \neq 0$.
Hence, $\vmin(a) = q$, and $\mathcal{B}_{n+1}$ is valuation independent.
\end{proof}

In the case $\ch{F} = 0$, we have a multiplicative
analogue.

\begin{theorem}
\label{valbasis:mult}
Let $n \geq 0$, and suppose that for each subset
$X$ of $I$ of cardinality at most $n$, we have a valuation
basis $B_X$ of $\U{K_X((G)) \cap L}$.
Suppose that
\[ \bigcup (B_X : X \subseteq I, |X| \leq n) \]
is valuation independent.
Then, for each subset $X$ of $I$ of cardinality $n+1$, we
may define a valuation basis $B_X$ of $\U{K_X((G)) \cap L}$
such that
\[ \bigcup (B_X : X \subseteq I, |X| \leq n+1) \]
is valuation independent.
\end{theorem}

\begin{proof}
As above, using Theorem~\ref{oap:add} instead of
\ref{oap:mult}.
\end{proof}

We can now prove Theorem~\ref{main-additive} easily.
\begin{proof}
By Theorem~\ref{valbasis:add}, we may take a valuation basis
$B_X$ of each valued $K$-vector space $\I{K_X((G)) \cap L}$
such that whenever $X' \subseteq X$, then $B_{X'} \subseteq B_X$.
It follows that the colimit of the $B_X$, over all finite subsets
$X$ of $I$, is a valuation basis for $\I{L}$.
\end{proof}

The proof of Theorem~\ref{main-multiplicative} is exactly analogous.

\section{Applications}
\label{applications}

Now, suppose that $F$ is real closed. Applying
Theorems~\ref{main-additive}, \ref{main-multiplicative}, and
\ref{alg tdrp} we immediately obtain:
\begin{corollary}
\label{valbasis:both}
Assume that $F$ is a real closed field, and $G$ a countable divisible
 ordered abelian group. There exist
$\mathbb{Q}$-valuation bases of
$(\I{\ac{F(G)}}, +)$ and $(\U{\ac{F(G)}}, \times)$
with respect to the minimal support valuation
$\vmin$.
\end{corollary}
If $F$ is archimedean, then the $\vmin$ valuation coincides with the
natural valuation on $F((G))$; we obtain
\begin{corollary}
Let $F$ be an archimedean real closed field,
and $G$ a countable divisible ordered
abelian group.
 Then
$(\I{\ac{K(G)}}, +)$ and $(\U{\ac{K(G)}}, \times)$
admit $\mathbb{Q}$-valuation bases with respect to the
natural valuation.
\end{corollary}

We can now obtain a partial answer to the original
question posed in the introduction.
Define the \emph{skeleton} of $V$ to be the
ordered system of $K$-vector spaces
$S(V):= [ \Gamma, \{ B(\gamma) \}_{\gamma \in \Gamma} ]$,
where the $\emph{component}$ $B(\gamma)$ is the $K$-vector space
\[ B(\gamma) = \{x\in V : v(x)\geq\gamma\} /
	\{x\in V : v(x)>\gamma\} \;. \]

Now, given an ordered system of $K$-vector spaces
$[ \Gamma, \{ B(\gamma) \}_{\gamma \in \Gamma} ]$, the product
$\prod_{\gamma\in\Gamma} B(\gamma)$ is a valued $K$-vector
space, where $\support(s)$ and $\vmin(s)$ are defined as for
fields of power series.
The Hahn sum $\coprod_{\gamma \in \Gamma} B(\gamma)$ is the subspace
of elements with finite support; its skeleton is precisely the
given system
$[ \Gamma, \{ B(\gamma) \}_{\gamma \in \Gamma} ]$.
By considering ``leading coefficients'', one sees that if $V$ has skeleton
$[ \Gamma, \{ B(\gamma) \}_{\gamma \in \Gamma} ]$ and admits
a valuation basis, then
$V \simeq \coprod_{\gamma \in \Gamma} B(\gamma)$.

\begin{corollary}
Let $F$ be an archimedean real closed field,
and $G$ a countable divisible ordered
abelian group.
Then the real closed field $\ac{F(G)}$
admits a restricted exponential.
\end{corollary}
\begin{proof}
Since $\I{\ac{F(G)}}$ and $\U{\ac{F(G)}}$ both admit
valuation bases, they are both isomorphic as ordered
abelian groups to the Hahn sums over their skeleta,
which are themselves isomorphic.
\end{proof}

Our final application is to the structure of complements to valuation
rings in fields of algebraic series.
Observe that for the field $F((G))$, an additive complement to the
valuation ring is given
by $F((G^{<0}))$, where $F((G^{<0}))$ is the (non-unital)
ring of power series with negative support. It follows easily (see
\cite{B-K-K})
that for the subfield $L=\ac{F(G)}$ of $F((G))$,
an additive complement to the valuation ring is given
by $\Neg(L)$, where $\Neg(K) := F((G^{<0})) \cap L$.
We shall call $\Neg(L)$ the \emph{canonical complement}
to the valuation ring of $L$.
Note that $F[G^{<0}] \subset \Neg(L)$,  where $F[G^{<0}]$ is
the \emph{semigroup ring} consisting of power series with negative
and finite support.
We are interested in understanding when $F[G^{<0}] = \Neg(L)$.
In [\cite{B-K-K}; Proposition~2.4], we proved the following
\begin{proposition}
Assume that $G$ is archimedean and divisible, and that $F$ is a real
closed field. Then $\Neg(L) = F[G^{<0}]$.
\end{proposition}
On the other hand, in [\cite{B-K-K}; Remark~2.5], we observed that
if $G$ is not archimedean, then $F[G^{<0}] \neq \Neg(L)$. The
results of this paper imply that:
\begin{proposition}
Let $L = \ac{F(G)}$, where $F$ is a real closed field and $G$ is a
countable divisible ordered abelian group. Then
$\Neg(L) \simeq F[G^{<0}]$.
\end{proposition}
\begin{proof}
We know that $L= \Neg(L) \oplus \vr{L}$, and this is a lexicographic
decomposition. Now the lexicographic sum of valued vector spaces admits
a valuation basis if and only if each summand admits a valuation basis
(see \cite{KS1}). It follows that $\Neg(L)$ admits a valuation basis.
Clearly $F[G^{<0}]$ also admits a valuation basis. Since $\Neg(L)$ and
$F[G^{<0}]$ have the same skeleton, it follows that they are isomorphic
as valued vector spaces, and in particular, as ordered groups under
addition.
\end{proof}
\bn\bn
\textbf{Acknowledgements:}
We wish to thank Antongiulio Fornasiero for contributing
to many insightful discussions and suggesting a proof of
Proposition~\ref{fg subring exists}.

\appendix
\section{Appendix}
\label{Appendix}
We prove our versions of our main theorems, weakened
under the assumption that the residue field of our power series
field has transcendence degree at most $\aleph_1$.
That is, we take $F$ to be an algebraically or real closed field and
assume that $\operatorname{tr deg} F \leq \aleph_1$; as in the body
of the paper, $G$ denotes a countable ordered abelian group.

In the body of the paper, we write $F$ as the colimit of countable
subfields of finite transcendence degree over $K$; the new assumption
$\operatorname{tr deg} F \leq \aleph_1$ enables us to additionally
assume this is a linear colimit over countable fields. The linearity
renders the prior combinatorial arguments (and supporting technical results)
unnecessary, as now we need only
verify the optimal approximation property for valued vector space
extensions of the form (in the additive case):
\[ \I{K_\lambda((G)) \cap L} \subseteq \I{K_{\lambda + 1}((G)) \cap L} \,. \]

In particular, we may fix a transcendence basis
$\{ \alpha_\lambda \}_{\lambda < \aleph_1}$ of $F$ over $K$.
Notice that the $\lambda < \aleph_1$ form a directed set.
For each $\lambda \leq \aleph_1$, define the subset
\[ X_\lambda = \{ \alpha_\gamma : \gamma < \lambda \} \]
and the corresponding subfield
\[ K_\lambda = \ac{K(X_\lambda)} \subset F \,. \]
where $\ac{\cdot}$ denotes relative algebraic closure in $F$.
Observe that we have the following colimits of countable objects:
\[ \varinjlim \lambda = \aleph_1 \quad
    \varinjlim K_\lambda = F \,.
\]
Moreover, given an intermediate field $F(G) \subseteq L \subseteq F((G))$
satisfying the TDRP over $K$, the first axiom implies
\[
\varinjlim L_\lambda = L \quad
    \varinjlim \I{L_\lambda} = \I{L} \quad
    \varinjlim \U{L_\lambda} = \U{L} \,.
\]
where $L_\lambda = K_\lambda((G)) \cap L$.

\begin{theorem}[Bounded Additive]
\label{main-ch-additive}
Let $F$ be an algebraically or real closed field
such that $\operatorname{tr deg} F \leq \aleph_1$,
$K$ a countably infinite subfield of $F$ and $G$ a countable ordered abelian
group.
If $F(G) \subseteq L \subseteq F((G))$ is an intermediate field
satisfying the TDRP over $K$, then the valued $K$-vector spaces
$(L, +)$ and therefore $(\I{L}, +)$ admit valuation bases.
\end{theorem}

\begin{proof}
For each $\lambda$, define the $K$-vector space
$V_\lambda = (L_\lambda, +)$.
We wish to define a valuation basis $B_\lambda$ for each
countable vector space $V_\lambda$ such that $B_{\lambda'}$ extends
$B_\lambda$ whenever $\lambda \prec \lambda'$.

First, we verify that $V_\lambda$ has the optimal approximation
property in $V_{\lambda + 1}$.
Indeed, suppose that $f \in V_{\lambda + 1} \setminus V_\lambda$;
by definition of $V_\lambda$, there exists a minimal
$q \in \support f$ such the power series coefficient $q(f)$ lies in
$K_{\lambda + 1} \setminus K_\lambda$.
Thus, if $h$ is any approximation to $f$ in $V_\lambda$, we
necessarily have $\vmin(f - h) \leq q$.

Assume for now that $F$ is algebraically closed.
By the second TDRP property, we may take an $K_\lambda$-rational
place $P$ of $K_{\lambda + 1}$ such that $\varphi_P(f) \in V_\lambda$;
it is then clear that $\varphi_P(f)$ is our desired optimal approximation.
On the other hand, if $F$ is real closed, we reduce to the previous
case --- if $f$ has an optimal approximation $g$ in
$V_\lambda \oplus \sqrt{-1} V_\lambda$,
then $(g + \sigma(g))/2$ is an optimal approximation to $g$ in $V_\lambda$,
where $\sigma$ is the non-trivial element of $\Gal{(L \oplus \sqrt{-1} L)/L}$.

Having established the optimal approximation property,
we are in a position to define the $B_\lambda$ via
transfinite induction.
For $\lambda = 0$, simply select an arbitrary valuation basis
$B_0$ of $V_0$.
For any successor ordinal $\lambda + 1$, note that
$V_{\lambda + 1}$ is countable and thus has countable dimension
over $V_\lambda$; hence, by Proposition~\ref{brown}, the valuation
basis $B_\lambda$ of $V_\lambda$ extends to one $B_{\lambda + 1}$
of $V_{\lambda + 1}$.
Now for a limit ordinal $\lambda$, we see that $V_\lambda$ is
the colimit of the $V_\rho$ for $\rho \prec \lambda$; hence,
we may simply define $B_\lambda$ to be the colimit of the
$B_\rho$ for $\rho \prec \lambda$.

Note that the constructed valuation basis for $V_{\aleph_1} = L$
is then simply $B_{\aleph_1}$.
\end{proof}

The proof of a multiplicative version is completely
analogous --- simply define $V_\lambda = (\U{L_\lambda}, \times)$
and replace the valuation $\vmin$ by $\vmin(1 - \cdot)$ in the above proof.
We thus have
\begin{theorem}[Bounded Multiplicative]
\label{main-ch-multiplicative}
Let $F$ be an algebraically or real closed field
of characteristic zero such that $\operatorname{tr deg} F \leq
\aleph_1$, and $G$ a countable ordered
abelian group.
If $F(G) \subseteq L \subseteq F((G))$ is an intermediate field
satisfying the TDRP over $\mathbb{Q}$ and the group $(\U{L}, \times)$
is divisible, then $(\U{L}, \times)$ is a valued
$\mathbb{Q}$-vector space and admits a $\mathbb{Q}$-valuation basis.
\end{theorem}

\adresse
\end{document}